\documentclass[12pt]{amsart}

\usepackage{amssymb}
\usepackage{amsbsy}
\usepackage{amscd}
\usepackage{amsmath}
\usepackage{amsthm}
\usepackage[mathscr]{eucal}

\newtheorem{theorem}{Theorem}[section]
\newtheorem{lemma}[theorem]{Lemma}
\numberwithin{defn}{section}
\newtheorem{remark}[theorem]{Remark}
\newtheorem{problem}[theorem]{Problem}
\numberwithin{equation}{section}

\newcommand{\cO}{{\mathcal O}}

\newcommand{\cE}{{\mathcal E}}
\newcommand{\cP}{{\mathcal P}}
\newcommand{\C}{\mathbb{C}}
\newcommand{\Z}{\mathbb{Z}}
\newcommand{\R}{\mathbb{R}}
\newcommand{\K}{\mathbb{K}}
\newcommand{\M}{\mathbb{M}}

\newcommand{\id}{\mathrm{id}}

\def\6t#1#2#3#4#5#6{\begin{CD}
#1@>>>#2 @>>>#3\\
@AAA@. @VV{\delta}V\\
#6@<<< #5 @<<<#4
\end{CD}}

\def\ex#1#2#3{
0\longrightarrow #1\longrightarrow 
#2\longrightarrow #3\longrightarrow 0}

\def\cp#1{(#1^{\otimes 2})\rtimes_\theta \Z_2}

\title[The $K$-theory of the flip automorphisms]{The $K$-theory of the flip automorphisms}

\author[M. Izumi]{Masaki Izumi}

\address{Department of Mathematics\\ Graduate School of Science\\
Kyoto University\\ Sakyo-ku, Kyoto 606-8502\\ Japan}

\email{izumi@math.kyoto-u.ac.jp}

\subjclass[2010]{46L80,19K99}

\keywords{flip, $K$-theory}

\begin{document}

\begin{abstract} We give an algorithm to compute the $K$-groups of the crossed product by the 
flip automorphism for a nuclear C$^*$-algebra satisfying the UCT. 
\end{abstract}

\maketitle

\section{Introduction} One of the most natural $\Z_2$-actions appearing 
in the theory of operator algebras arises from the flip automorphism, 
and it often characterizes important classes of operator algebras 
(see \cite{S75}, \cite{Co76}, \cite{ER78}, \cite{KP00}, \cite{TW07}). 
More precisely, the flip automorphism acting on the tensor square $A^{\otimes 2}=A\otimes A$ of a C$^*$-algebra $A$, 
is an extension of the map $x\otimes y\mapsto y\otimes x$. 

In \cite{I041}, \cite{I042}, the author developed the theory of finite group actions with the Rohlin property, 
and showed that the flip automorphism in the case of the Cuntz algebra $\cO_2$ has the Rohlin property. 
The first step to obtain this result was to compute the $K$-groups of the fixed point algebra, or equivalently, 
the crossed product. 
Since the flip automorphism is available for every C$^*$-algebra, it would be desirable to have a systematic 
algorithm to compute the $K$-groups of the crossed product by the flip automorphism. 
The purpose of this note is to provide such an algorithm in the case of separable nuclear C$^*$-algebras 
satisfying the UCT. 
The author would not be surprised if some of the contents of this note is known to specialists, 
but he believes that they are still worth to be written. 

This note is an extended version of the author's private note written around 2001. 
The author would like to thank Hiroki Matui for useful discussions, and the anonymous referee for careful reading. 
This work is supported in part by JSPS KAKENHI Grant Number JP15H03623. 

\section{$KK^{\Z_2}$-equivalence} 
Let $A$ be a C$^*$-algebra, and let $n$ be a natural number greater than 1. 
Throughout this note, 
we denote by $\theta_A$, or $\theta$ for simplicity, the flip automorphism of $A^{\otimes 2}$. 
We denote by $SA$ the suspension of $A$, which is identified with either $C_0((0,1),A)$ or $C_0(\R,A)$ 
depending on the situation. 
We denote by $\M_n$ and $\K$ the $n$ by $n$ matrix algebra and the set of compact operators on a separable 
infinite dimensional Hilbert space respectively. 
The symbol $\Z_n$ stands for the cyclic group $\Z/n\Z$. 
We often use the fact that the crossed product $(A\oplus A)\rtimes_\gamma \Z_2$ is isomorphic to $\M_2(A)$, 
where $\gamma(x\oplus y)=y\oplus x$.

\begin{theorem}\label{KK} Let $A$ and $B$ be mutually $KK$-equivalent separable nuclear $C^*$-algebras. 
Then $A^{\otimes2}$ and $B^{\otimes2}$ are $KK^{\Z_2}$-equivalent to each other with $\Z_2$-actions given 
by the  flip automorphisms. 
\end{theorem}

\begin{proof} Let $S^2A=S(SA)$, which is identified with $C_0(\C)\otimes A$. 
We first claim that $(S^2A)^{\otimes2}$ is $KK^{\Z_2}$-equivalent to $A^{\otimes2}$. 
Indeed, we note that the flip $\theta_{S^2A}$ acting on $(S^2A)^{\otimes2}$ is conjugate to 
$\theta_{C_0(\C)}\otimes \theta_A$ acting on $C_0(\C)^{\otimes 2}\otimes A^{\otimes 2}$, 
and that $\theta_{C_0(\C)}$ on $C_0(\C)^{\otimes2}=C_0(\C^2)$ comes from the complex linear map $\left(
\begin{array}{cc}
0 &1  \\
1 &0 
\end{array}
\right)$ of the underlaying space $\C^2$.  
Thus thanks to \cite[Theorem 20.3.2]{B98}, we get the claim. 

Let $\phi_t:S^2A\otimes \K\to S^2B\otimes \K$ and 
$\psi_t:S^2B\otimes \K\to S^2A\otimes \K$ be completely positive asymptotic 
morphisms giving the $KK$-equivalence of $A$ and $B$. 
Then $\phi_t\otimes \phi_t$ and $\psi_t\otimes \psi_t$ gives $KK^{\Z_2}$-equivalence of 
$(S^2A\otimes \K)^{\otimes2}$ with $\theta_{S^2A\otimes \K}$ and $(S^2B\otimes \K)^{\otimes2}$ with 
$\theta_{S^2B\otimes \K}$. 
Therefore we get the statement. 
\end{proof}

Our idea to compute $K_*(\cp A)$ for a given separable nuclear C$^*$-algebra $A$ is as follows. 
We choose an appropriate $B$ that is $KK$-equivalent to $A$, and compute $K_*(\cp A)$ from $K_*(\cp B)$ 
using Theorem \ref{KK}. 
If $B=\bigoplus_{i=1}^nB_i$, we have 
\begin{align*}
\lefteqn{\cp B} \\
 &\cong \bigoplus_{i=1}^n\cp {B_i}\oplus \bigoplus_{1\leq i<j\leq n}(B_i\otimes B_j\oplus B_j\otimes B_i)\rtimes \Z_2\\
 &\cong \bigoplus_{i=1}^n\cp {B_i}\oplus \bigoplus_{1\leq i<j\leq n}\M_2(B_i\otimes B_j),
\end{align*}
and 
$$K_*(\cp A)\cong \bigoplus_{i=1}^nK_*(\cp {B_i})\oplus \bigoplus_{1\leq i<j\leq n}K_*(B_i\otimes B_j).$$
The dual action $\widehat{\theta_B}$ acts on $K_*(\M_2(B_i\otimes B_j))$ trivially. 

Assume moreover that $A$  satisfies the UCT and its $K$-groups are finitely generated. 
Then we can choose each $B_i$ isomorphic to one of the following building blocks: $\C$, the Cuntz algebra $\cO_{n+1}$, $S\C$, and the dimension drop 
algebra 
$$D_n=\{f\in C_0([0,1),\M_n);\; f(0)\in \C1_{\M_n}\}.$$
Note that their $K$-groups are 
$$K_*(\C)\cong\left\{
\begin{array}{ll}
\Z , &\quad *=0 \\
\{0\} , &\quad *=1
\end{array}
\right.,$$
$$K_*(\cO_{n+1})\cong\left\{
\begin{array}{ll}
\Z_n , &\quad *=0 \\
\{0\} , &\quad *=1
\end{array}
\right.,$$
$$K_*(S\C)\cong\left\{
\begin{array}{ll}
\{0\} , &\quad *=0 \\
\Z , &\quad *=1
\end{array}
\right.,$$
$$K_*(D_n)\cong\left\{
\begin{array}{ll}
\{0\} , &\quad *=0 \\
\Z_n , &\quad *=1
\end{array}
\right..$$

\section{Building bocks}
In this section, we compute $K_*(\cp A)$ for each of the building blocks. 

We denote by $\lambda$ the implementing unitary for $\theta$ in the crossed product $\cp A$, 
and by $e_+$ and $e_-$ the spectral projections of $\lambda$ corresponding to the eigenvalues 
$1$ and $-1$ respectively.  
We denote by $\delta$ the exponential map in the 6-term exact sequence of the $K$-groups.

\subsection{$\cp{\C}$} 

It is trivial to see 
$$K_*(\cp \C)=K_*(\C \Z_2)\cong \left\{
\begin{array}{ll}
\Z^2 , &\quad *=0 \\
\{0\} , &\quad *=1
\end{array}
\right.,
$$
and the dual action $\hat{\theta}$ acts on $\Z^2$ as the flip of the two components. 

Let $\cO_\infty$ be the Cuntz algebra, which is the universal C$^*$-algebra generated by isometries $\{S_i\}_{i=1}^\infty$ 
with mutually orthogonal ranges. 
It is well known that $\cO_\infty$ is $KK$-equivalent to $\C$, and Theorem \ref{KK} implies that 
$\cO_\infty^{\otimes 2}$ with the flip automorphism is $KK^{\Z_2}$-equivalent 
to $\C$ with a trivial action, which is further known to be $KK^{\Z_2}$-equivalent to 
$\cO_\infty$ with a quasi-free action (see \cite{Pi97} for example). 
On the other hand, Goldstein and the author \cite{GI11} showed that all non-trivial quasi-free $\Z_2$-actions 
on $\cO_\infty$ are mutually conjugate.  
Note that $\cO_\infty\otimes \cO_\infty$ is isomorphic to $\cO_\infty$ thanks to 
Kirchberg-Phillips \cite{KP00}. 

\begin{problem} Is the flip on $\cO_\infty^{\otimes 2}$ conjugate to a quasi-free $\Z_2$-action on $\cO_\infty$ ? 
\end{problem}

To solve the above problem affirmatively, it suffices to show that the flip on $\cO_\infty^{\otimes 2}$ is 
approximately representable as a $\Z_2$-action (see \cite{GI11}). 
More generally we can pose the following problem. 

\begin{problem} Let $D$ be a strongly self-absorbing C$^*$-algebra with sufficiently many projections (see \cite{TW07}). 
Is the flip on $D^{\otimes 2}$ approximately (asymptotically) representable ?
\end{problem} 

Note that there is a unique $\Z_2$-action on the Cuntz algebra $\cO_2$ with the Rohlin property (see \cite{I041}), 
and it is at the same time asymptotically representable. 
Thus the second problem is affirmative for $\cO_2$. 
Since the flip on $\M_n^{\otimes 2}$ is inner, the flip on any UHF algebra is approximately representable. 

\subsection{$\cp{\cO_{n+1}}$} 

Let $\cO_{n+1}$ be the Cuntz algebra, which is the universal C$^*$-algebra generated by $n+1$ isometries 
$\{S_i\}_{i=0}^{n}$ with mutually orthogonal ranges satisfying $\sum_{i=0}^nS_iS_i^*=1$. 

Our computation in the case of $\cp{\cO_{n+1}}$ and $\cp{D_n}$ is based on the following 
elementary fact. 
\begin{lemma}\label{ED} The elementary divisors of an integer matrix 
$$\left(
\begin{array}{cc}
-n &\frac{n^2-n}{2}  \\
-n &\frac{n^2+n}{2} 
\end{array}
\right)$$
is $(n,n)$ for odd $n$ and $(\frac{n}{2},2n)$ for even $n$. 
\end{lemma}

\begin{theorem}
Let the notation be as above. \\
$(1)$ When $n$ is odd, 
$$K_0(\cp{\cO_{n+1}})
\cong \Z_n\oplus\Z_n,$$
$$K_1(\cp{\cO_{n+1}})
\cong \{0\}.$$
The canonical generators of the right-hand side for  
$K_0(\cp{\cO_{n+1}})$ 
are  given by $[e_+]$, $[e_-]$. \\
$(2)$ When $n$ is even, 
$$K_0(\cp{\cO_{n+1}})
\cong \Z_{2n}\oplus\Z_{n/2},$$
$$K_1(\cp{\cO_{n+1}})
\cong \{0\}.$$
The canonical generators of the right-hand side for  
$K_0(\cp{\cO_{n+1}})$ 
are  given by $[e_+]$, $[e_+]-[e_-]$.
\end{theorem} 

\begin{proof} Let $\cE_{n+1}$ be the Cuntz-Toeplitz algebra, which is the universal C$^*$-algebra  
generated by $n+1$ isometries $\{T_i\}_{i=0}^n$ with mutually orthogonal ranges. 
Let $p\in \cE_{n+1}$ be the projection defined by  
$$p=1-\sum_{i=0}^n T_iT_i^*.$$
Then the closed ideal of $\cE_{n+1}$ generated by $p$ is identified with $\K$, and $\cO_{n+1}$ is identified with 
the quotient $\cE_{n+1}/\K$. 

We define a closed ideal $J$ of 
$\cE_{n+1}\otimes \cE_{n+1}$ by 
$$J=\K\otimes \cE_{n+1}+\cE_{n+1}\otimes \K.$$
Then we have two short exact sequences  
$$\ex{J}{\cE_{n+1}\otimes \cE_{n+1}}
{\cO_{n+1}\otimes \cO_{n+1}},$$
$$\ex{\K\otimes \K}{J}
{(\K\otimes \cO_{n+1})\oplus (\cO_{n+1}\otimes \K)},$$
and two short exact sequences
\begin{equation}\label{C1}
\ex{J\rtimes_\theta\Z_2}
{\cp{\cE_{n+1}}}
{\cp{\cO_{n+1}}},
\end{equation}
\begin{equation}\label{C2}
\ex{\cp{\K}}{J\rtimes_\theta\Z_2}
{\M_2(\K\otimes \cO_{n+1})}.
\end{equation}
Using (\ref{C2}), we get the 6-term exact sequence
$$\6t{\Z^2}{K_0(J\rtimes_\theta \Z_2)}{\Z_n}
{0}{K_1(J\rtimes_\theta \Z_2)}{0}\;,$$
which implies $K_1(J\rtimes_\theta \Z_2)=\{0\}$. 
The generators of $\Z^2$ and $\Z_n$ above are 
$[(p\otimes p)e_+]$, $[(p\otimes p)e_-]$ and 
$[p\otimes 1]$ respectively. 
In $K_0(J\rtimes_\theta \Z_2)$, we have 
$$[(p\otimes p)e_+]+[(p\otimes p)e_-]=[p\otimes p]
=-n[p\otimes 1],$$
which shows that $K_0(J\rtimes \Z_2)$ is isomorphic to $\Z^2$ with 
generators $[p\otimes 1]$ and $[(p\otimes p)e_+]$. 

(\ref{C1}) implies the 6-term exact sequence 
$$\6t{\Z^2}{\Z^2}
{K_0(\cp{\cO_{n+1}})}
{0}{0}{K_1(\cp{\cO_{n+1}})},$$
where we use the facts
$$K_0(\cp{\cE_{n+1}})=\langle [e_+],[e_-]\rangle\cong \Z^2,$$
$$K_1(\cp{\cE_{n+1}})\cong \{0\},$$
which follows from Theorem \ref{KK} as $\cE_{n+1}$ is $KK$-equivalent to $\C$ (see \cite{Pi97} for example). 

Now we compute the image of the generators of $K_0(J\rtimes_\theta \Z_2)$ 
in $K_0(\cp{\cE_{n+1}})$. 
It is immediate to get 
$$[p\otimes 1]=-n[1]=-n([e_+]+[e_-]).$$
To finish the proof, it suffices to show
$$[(p\otimes p)e_+]=\frac{n^2-n}{2}[e_+]+\frac{n^2+n}{2}[e_-],$$
as it implies 
$$K_0(\cp{\cO_{n+1}})\cong \mathrm{coker}(\left(
\begin{array}{cc}
-n &\frac{n^2-n}{2}  \\
-n &\frac{n^2+n}{2} 
\end{array}
\right):\Z^2\rightarrow \Z^2),$$
$$K_1(\cp{\cO_{n+1}})\cong \ker(\left(
\begin{array}{cc}
-n &\frac{n^2-n}{2}  \\
-n &\frac{n^2+n}{2} 
\end{array}
\right):\Z^2\rightarrow \Z^2),$$
and the theorem follows from Lemma \ref{ED}.

We have 
\begin{align*}
\lefteqn{(p\otimes p)e_+} \\
 &=e_+-\sum_{i=0}^n (1\otimes T_iT_i^*+T_iT_i^*\otimes 1
-T_iT_i^*\otimes T_iT_i^*)e_+\\
&+\sum_{i<j}(T_iT_i^*\otimes T_jT_j^*+T_jT_j^*\otimes T_iT_i^*)
e_+\\
&=e_+-\sum_{i=0}^n (T_iT_i^*\otimes T_iT_i^*)e_+\\
&-\sum_{i=0}^n \bigl((1-T_iT_i^*)\otimes T_iT_i^*
+T_iT_i^*\otimes (1-T_iT_i^*)\bigr)e_+\\
&+\sum_{i<j}(T_iT_i^*\otimes T_jT_j^*+T_jT_j^*\otimes T_iT_i^*)
e_+.
\end{align*}
For $0\leq i\leq n$, we set  
$$P_i=\bigl((1-T_iT_i^*)\otimes T_iT_i^*
+T_iT_i^*\otimes (1-T_iT_i^*)\bigr)e_+,$$
which are mutually commuting projections. 
For distinct $i,j,k$, we have $P_iP_jP_k=0$, and 
$$P_iP_j=(T_iT_i^*\otimes T_jT_j^*+T_jT_j^*\otimes T_iT_i^*)
e_+.$$
Thus \begin{align*}
\lefteqn{\sum_{i=0}^n \bigl((1-T_iT_i^*)\otimes T_iT_i^*
+T_iT_i^*\otimes (1-T_iT_i^*)\bigr)e_+} \\
 &-\sum_{i<j}(T_iT_i^*\otimes T_jT_j^*+T_jT_j^*\otimes T_iT_i^*)e_+\\
 &=\sum_{i=0}^{n-1}(P_i-\sum_{j=i+1}^nP_iP_j)+P_n
\end{align*}
is a projection orthogonal to $(T_iT_i^*\otimes T_iT_i^*)e_+$, 
and 
\begin{align*}
\lefteqn{[(p\otimes p)e_+]} \\
 &=[e_+]-\sum_{i=0}^n [(T_iT_i^*\otimes T_iT_i^*)e_+]-\sum_{i=0}^{n-1}[P_i-\sum_{j=i+1}^nP_iP_j]-[P_n] \\
 &=-n[e_+] -\sum_{i=0}^n [P_i]+\sum_{i<j}[P_iP_j].
\end{align*}

We set 
\begin{align*}
\lefteqn{
V_i=\frac{(1-T_iT_i^*)\otimes T_i+T_i\otimes(1-T_iT_i^*)}{\sqrt{2}}e_+}\\ 
&+\frac{(1-T_iT_i^*)\otimes T_i-T_i\otimes(1-T_iT_i^*)}{\sqrt{2}}e_-.
\end{align*}
Then $V_iV_i^*=P_i$
and 
\begin{align*}
V_i^*V_i&=\frac{(1-T_iT_i^*)\otimes 1+1\otimes(1-T_iT_i^*)}{2}e_+\\
 &+\frac{(1-T_iT_i^*)\otimes 1+1\otimes(1-T_iT_i^*)}{2}e_-\\
 &+\frac{(1-T_iT_i^*)\otimes 1-1\otimes(1-T_iT_i^*)}{2}e_+\\
 &+\frac{(1-T_iT_i^*)\otimes 1-1\otimes(1-T_iT_i^*)}{2}e_-\\
 &=(1-T_iT_i^*)\otimes 1.
\end{align*}
Thus we get $[P_i]=0$. 

For $i\neq j$ we define 
$$V_{ij}=\frac{T_i\otimes T_j+T_j\otimes T_i}{\sqrt{2}}e_+
+\frac{T_i\otimes T_j-T_j\otimes T_i}{\sqrt{2}}e_-.$$
Then direct computation yields
$V_{ij}V_{ij}^*=P_iP_j$ and $V_{ij}^*V_{ij}=1$, 
which shows 
$$[P_iP_j]=[1].$$
Thus we get 
$$[(p\otimes p)e_+]=-n[e_+]+\frac{n(n+1)}{2}[1]
=\frac{n^2-n}{2}[e_+]+\frac{n^2+n}{2}[e_-],$$
which finishes the proof. 
\end{proof}

\subsection{$\cp {(S\C)}$}

\begin{theorem}\label{SC} With the above notation, we have 
$$K_0(\cp{(S\C)})=\{0\},$$
$$K_1(\cp{(S\C)})=\Z.$$
The flip $\theta$ acts on $K_0((S\C)^{\otimes 2})\cong \Z$ by $-1$ and the dual action $\hat{\theta}$ acts on 
$K_1(\cp{(S\C)})\cong \Z$ by $-1$. 
\end{theorem}

\begin{proof}
It suffices to work on $C_0(\R)\rtimes_\alpha\Z_2$, 
where the action $\alpha$ is given by $\alpha(f)(t)=f(-t)$,
because we have 
$$\cp{(S\C)}\cong S(C_0(\R)\rtimes_\alpha\Z_2).$$
Considering the evaluation map at 0, we get the short exact sequence 
$$\ex{C_0((0,\infty))\oplus C_0((-\infty,0))}{C_0(\R)}{\C}.$$
and
$$0\longrightarrow \M_2(C_0((0,\infty)))\longrightarrow 
C_0(\R)\rtimes_\alpha \Z_2\longrightarrow 
\C\Z_2\longrightarrow 0,$$
which implies the 6-term exact sequence
$$\6t
{0}{K_0(C_0(\R)\rtimes_\alpha \Z_2)}{\Z^2}
{\Z}{K_1(C_0(\R)\rtimes_\alpha \Z_2)}{0}
\quad.$$
Let $f\in C_0(\R)$ be a real function satisfying 
$f(0)=1$, $f(x)=f(-x)$. 
Then, we have 
$$\delta([e_+])=[e^{2\pi i f}e_++e_-]=[e^{2\pi i f}e_-+e_+]
=\delta([e_-]),$$
which is the generator of $K_1(\M_2(C_0(\R_{>0})))$.
Thus 
$$K_0(C_0(\R)\rtimes_\alpha \Z_2)\cong\Z,$$
$$K_1(C_0(\R)\rtimes_\alpha \Z_2)\cong\{0\}.$$
The generator of $K_0(C_0(\R)\rtimes_\alpha \Z_2)$ is the 
preimage of $[e_+]-[e_-]$, on which the dual action $\hat{\theta}$ acts by $-1$.   
\end{proof}

\begin{remark}
The generator of $K_0(C_0(\R)\rtimes_\alpha \Z_2)$ is the preimage of $[e_+]-[e_-]=2[e_+]-[1]$, 
and it can be written down explicitly as follows. 
Let $h$ be continuous monotone increasing real function on $\R$ 
satisfying $h(-x)=-h(x)$ and 
$$\lim_{t\to\pm\infty}h(t)=\pm\frac{\pi}{2}.$$
We set $c(t)=\cos h(t)$, $s(t)=\sin h(t)$. 
Then, the generator is 
$$[
\left(\begin{array}{cc}
\frac{1+sc+c^2\lambda}{2}&\frac{s^2-sc\lambda}{2}\\
\frac{s^2-\lambda sc}{2} &\frac{1-sc+c^2\lambda}{2}
\end{array}\right)]-[1].$$
\end{remark}

For later use, we give a more detailed statement actually proved in the proof of Theorem \ref{SC}. 

\begin{lemma} \label{restriction} Let 
$$R:\cp{(S\C)}\to (S\C)\rtimes_{\id} \Z_2=S(\C\Z_2)$$ be the homomorphism arising from the restriction map to the diagonal set, 
and let $\eta$ be the natural isomorphism from $K_1(S(\C\Z_2))$ onto $K_0(\C\Z_2)\cong \Z^2$. 
Then the map $\eta\circ R_*:K_1(\cp{(S\C)})\to K_0(\C\Z_2)$ is injective and 
its image is $\Z([e_+]-[e_-])$. 
\end{lemma}

There exists a unique separable unital simple purely infinite C$^*$-algebra $KK$-equivalent to $S\C$, 
which is denoted by $\cP_\infty$. 
Its tensor square $\cP_\infty\otimes \cP_\infty$ is Morita equivalent to $\cO_\infty$, and 
$\theta$ acts on $K_0(\cP_\infty^{\otimes 2})\cong \Z$ non-trivially. 
Theorem \ref{KK} and Theorem \ref{SC} imply that $\cp{\cP_\infty}$ is isomorphic to $\cP_\infty$, 
and the dual action $\hat{\theta}$ acts on $K_1(\cp{\cP_\infty})$ non-trivially. 

\begin{problem} Let $A$ be a unital C$^*$-algebra Morita equivalent to $\cO_\infty$ in the Cuntz standard form, 
that is, the $K_0$-class of the identity is 0. 
Let $\alpha$ be a $\Z_2$-action on $A$ acting on $K_0(A)$ non-trivially. 
Can we say something about the isomorphism class of $A\rtimes_\alpha \Z_2$ ? 
Is $A\rtimes_\alpha \Z_2$ isomorphic to $\cP_\infty$ ?
\end{problem}

\subsection{$\cp{D_n}$}
Let $D_n$ be the dimension drop algebra, that is, 
$$D_n=\{f\in C_0([0,1),\M_n);\; f(0)\in \C1\}.$$
We identify $S\M_n$ with $\{f\in D_n;\; f(0)=0\}$. 
We denote by $\{e_{ij}\}_{1\leq i,j\leq n}$ the canonical system of matrix units of $\M_n$. 
There exists a short exact sequence
$$\ex{S\M_n}{D_n}{\C},$$
which implies 
$$K_0(D_n)\cong\{0\},$$
$$K_1(D_n)=\langle u \rangle\cong \Z_n,$$ 
where $u(x)=e^{2\pi ix}e_{11}+1-e_{11}.$ 

\begin{theorem} Let the notations be as above. \\
$(1)$ When $n$ is odd, 
$$K_0(\cp{D_n})\cong \{0\},$$
$$K_1(\cp{D_n})\cong \Z_n\oplus \Z_n.$$
$(2)$ When $n$ is even, 
$$K_0(\cp{D_n})\cong \{0\},$$
$$K_1(\cp{D_n})\cong \Z_{2n}\oplus \Z_{n/2}.$$
\end{theorem}

\begin{proof} We realize $D_n\otimes D_n$ inside $C([0,1]^2,\M_n\otimes \M_n)$, and define 
a closed ideal $J$ of $D_n\otimes D_n$ by 
$$J=S\M_n\otimes D_n+D_n\otimes S\M_n
=\{f\in D_n\otimes D_n;\; f(0,0)=0\}.$$
Then we have two short exact sequences 
$$\ex{J}{D_n\otimes D_n}{\C},$$
$$\ex{S\M_n\otimes S\M_n}{J}{S\M_n\otimes \C\oplus\C\otimes S\M_n},$$
which imply the two short exact sequences
\begin{equation}\label{D1}
\ex{J\rtimes_\theta \Z_2}{\cp{D_n}}{\C\Z_2},
\end{equation}
\begin{equation}\label{D2}
\ex{\cp{(S\M_n)}}{J\rtimes_\theta \Z_2}
{S\M_{2n}}.
\end{equation}
From (\ref{D2}) and Theorem \ref{SC}, we get the 6-term exact sequence 
\begin{equation}\label{D3}
\6t{0}{K_0(J\rtimes_\theta \Z_2)}{0}
{\Z}{K_1(J\rtimes_\theta \Z_2)}{\Z}\quad ,
\end{equation}
and 
$$K_0(J\rtimes_\theta \Z_2)\cong \{0\},$$
$$K_1(J\rtimes_\theta \Z_2)\cong \Z^2.$$

We determine the generators of the latter now. 
Let $\iota$  be the inclusion map $\iota:\cp{(S\M_n)}\to J\rtimes_\theta \Z_2$ and let $\pi$ be quotient map 
$$\pi: J\rtimes_\theta \Z_2\to (S\M_n\otimes \C\oplus\C\otimes S\M_n)\rtimes_\theta \Z_2\cong \M_2(S\M_n)
\cong S\M_{2n}.$$ 
Note that $\pi$ restricted to $J$ is given by the direct sum of the restriction map 
to $y=0$ and that of $x=0$. 
We set $w=Fe_++e_-$ with $F\in J+1_n$ as follows: 
$F$ is a unitary that is invariant under $\theta$, and for $0\leq y\leq x\leq 1$, 
$(x,y)\neq (0,0)$, it is given by 
\begin{align*}
\lefteqn{F(x,y)} \\
 &=z(x)e_{11}\otimes e_{11}+(1-e_{11})\otimes (1-e_{11}) \\
 &+(z(x)-1)\sum_{j=2}^n 
(c(x,y)^2e_{11}\otimes e_{jj}+s(x,y)^2e_{jj}\otimes e_{11}\\
&+c(x,y)s(x,y)e_{1j}\otimes e_{j1}+c(x,y)s(x,y)e_{j1}\otimes e_{1j}) \\
 &+\sum_{j=2}^n 
(e_{11}\otimes e_{jj}+e_{jj}\otimes e_{11}) \\
 &=z(x)e_{11}\otimes e_{11}+(1-e_{11})\otimes (1-e_{11})\\
&+z(x)\sum_{j=2}^n 
(c(x,y)^2e_{11}\otimes e_{jj}+s(x,y)^2e_{jj}\otimes e_{11}\\
&+c(x,y)s(x,y)e_{1j}\otimes e_{j1}+c(x,y)s(x,y)e_{j1}\otimes e_{1j})\\
&+\sum_{j=2}^n 
(s(x,y)^2e_{11}\otimes e_{jj}+c(x,y)^2e_{jj}\otimes e_{11}\\
&-c(x,y)s(x,y)e_{1j}\otimes e_{j1}-c(x,y)s(x,y)e_{j1}\otimes e_{1j}),
\end{align*}
where $z(x)=e^{2\pi i x},$
$$c(x,y)=\cos \frac{\pi y}{4x},$$
$$s(x,y)=\sin \frac{\pi y}{4x},$$ 
Then $[w]\in K_1(J\rtimes_\theta \Z_2)$ satisfies
$$\pi(w)=(ze_{11}\oplus ze_{11})e_++1-(e_{11}\oplus e_{11})e_+,$$
which is a generator of 
$K_1((S\M_n\otimes \C\oplus\C\otimes S\M_n)\rtimes_\theta \Z_2).$

Let $D:J\rtimes_\theta\Z_2\longrightarrow S(\cp{\M_n})$ be the homomorphism arising from 
the restriction map to the diagonal set, 
and let $D^0$ be the restriction of $D$ to $\cp{(S\M_n)}$. 
Note that the flip on $\M_n^{\otimes 2}$ is the 
inner automorphism induced by a unitary $\cE\in \M_n\otimes \M_n$ given by   
$$\cE=\sum_{ij}e_{ij}\otimes e_{ji}.$$
Thus $\cp{\M_n}$ is isomorphic to $\M_n^{\otimes 2}\oplus \M_n^{\otimes 2}$, and 
the two minimal central projections of $\cp{\M_n}$ are 
$$e_+'=\frac{1+\lambda'}{2},\quad 
e_-'=\frac{1-\lambda'}{2},$$
where $\lambda'=\cE\lambda$. 
Thus the canonical generators of $K_0(\cp{\M_n})\cong\Z^2$ are 
$[(e_{11}\otimes e_{11})e'_{+}]$, $[(e_{11}\otimes e_{11})e'_{-}]$. 
We denote by $\eta$ the natural isomorphism from 
$K_1(S(\cp{\M_n}))$ onto $K_0(\cp{\M_n})$. 
The $K$-groups of $\cp{(S\M_n)}$ can be computed as in Theorem \ref{SC}, 
and Lemma \ref{restriction} implies that 
$$\eta\circ D^0_* :K_1(\cp{(S\M_n)})\longrightarrow K_0(\cp{\M_n})$$ 
is an injective map whose image is 
$$\Z ([(e_{11}\otimes e_{11})e'_+]-[(e_{11}\otimes e_{11})e'_-]).$$ 
We choose an element $g\in K_1(\cp{(S\M_n)})$ satisfying 
$$\eta\circ D^0_*(g)=
[(e_{11}\otimes e_{11})e'_+]-[(e_{11}\otimes e_{11})e'_-].$$
Then (\ref{D3}) shows that $\{[w], \iota_*(g)\}$ generate 
$K_1(J\rtimes_\theta \Z_2)\cong \Z^2$. 

Now (\ref{D1}) implies the 6-term exact sequence  
\begin{equation}\label{D4}
\6t{0}{K_0(\cp{D_n})}{K_0(\C\Z_2)\cong\Z^2}{K_1(J\rtimes_\theta \Z_2)\cong \Z^2}{K_1(\cp{D_n})}{0}\quad . 
\end{equation}
We show 
\begin{equation}\label{ED1}
\delta([e_+])=-n[w]+\frac{n(n-1)}{2}\iota_*(g),\end{equation}
\begin{equation}\label{ED2}
\delta([e_-])=-n[w]+\frac{n(n+1)}{2}\iota_*(g),\end{equation}
which together with Lemma \ref{ED} finishes the proof. 

Let $f(x,y)=(1-x)(1-y)$. 
Then $f\in D_n\otimes D_n$, and 
$$\delta([e_+])=[e^{2\pi i f}e_++e_-]\in K_1(\cp J),$$ 
$$\delta([e_-])=[e^{2\pi i f}e_-+e_+]\in K_1(\cp J).$$
Thus  
$$\pi_*\circ\delta([e_+])=\pi_*\circ\delta([e_-])=-n\pi_*([w]),$$
and there exist integers $k$, $l$ such that 
$$\delta([e_+])=-n[w]+k\iota_*(g),$$
$$\delta([e_-])=-n[w]+l\iota_*(g).$$
Since $\eta\circ D_*\circ \iota_*=\eta\circ D^0_*$ and $D^0_*$ 
is injective, 
the two numbers $k$ and $l$ are determine by 
$$\eta\circ D_*\circ \delta([e_+])=-n\eta\circ D_*([w])+k\eta\circ D_*^0(g),$$
$$\eta\circ D_*\circ \delta([e_-])=-n\eta\circ D_*([w])+l\eta\circ D_*^0(g),$$
or equivalently, by 
\begin{equation}\label{ED3}
\eta\circ D_*\circ \delta([e_+])+n\eta\circ D_*([w])=k
([(e_{11}\otimes e_{11})e'_+]-[(e_{11}\otimes e_{11})e'_-]),\end{equation}
\begin{equation}\label{ED4}
\eta\circ D_*\circ \delta([e_-])+n\eta\circ D_*([w])=l
([(e_{11}\otimes e_{11})e'_+]-[(e_{11}\otimes e_{11})e'_-]).\end{equation}
Now our task is to compute the left-hand sides of Eq.(\ref{ED3}) and Eq.(\ref{ED4}). 

The first terms are
\begin{equation}\label{ED5}
\eta\circ D_*\circ \delta([e_+])=\eta\circ D_*([e^{2\pi if}e_++e_-])=-[e_+],\end{equation}
\begin{equation}\label{ED6} \eta\circ D_*\circ \delta([e_-])=\eta\circ D_*([e^{2\pi if}e_-+e_+])=-[e_-].
\end{equation}
Letting $p_j=(e_{11}\otimes e_{jj}+e_{jj}\otimes e_{11}
+e_{1j}\otimes e_{j1}+e_{j1}\otimes e_{1j})/2$, we get the second terms as   
\begin{equation}\label{ED7} 
\eta\circ D_*([w])=[(e_{11}\otimes e_{11})e_+]
+\sum_{j=2}^n[p_je_+].\end{equation}
Using the fact that two projections $\frac{1+\cE}{2}$ and $\frac{1-\cE}{2}$ have rank 
$\frac{n(n+1)}{2}$ and $\frac{n(n-1)}{2}$ respectively, we get the following equality in $K_0(\cp{\M_n})$:  
\begin{align*}
[e_+]&=[\frac{1+\cE\lambda'}{2}]
=[\frac{1+\cE}{2}e_+']+[\frac{1-\cE}{2}e_-']\\
&=\frac{n(n+1)}{2}[(e_{11}\otimes e_{11})e_+']
+\frac{n(n-1)}{2}[(e_{11}\otimes e_{11})e_-'].
\end{align*}
In a similar way, we get 
$$[e_-]=\frac{n(n-1)}{2}[(e_{11}\otimes e_{11})e'_+]
+\frac{n(n+1)}{2}[(e_{11}\otimes e_{11})e'_-],$$
$$[p_je_+]=[(e_{11}\otimes e_{11})e'_+].$$
Therefore from Eq.(\ref{ED3})-(\ref{ED7}), 
we obtain $k=\frac{n^2-n}{2}$, $l=\frac{n^2-n}{2}$, and Eq.(\ref{ED1}), Eq.(\ref{ED2}). 
\end{proof}

\begin{remark} In the above, the dual action $\hat{\theta}$ acts on the generators $\iota_*(g)$ and $[w]$ of 
$K_1({D_n})$ as $\hat{\theta}_*(\iota_*(g))=-\iota_*(g)$ and $\hat{\theta}_*([w])=[w]+\iota_*(g)$. 
\end{remark}


\end{document}